\documentclass[12pt]{article}

\usepackage{amsmath,amssymb,latexsym,amsfonts,mathrsfs,graphicx}
\usepackage{algpseudocode,algorithm,algorithmicx,bm}
\usepackage{fullwidth}
\usepackage{placeins}
\usepackage[margin=1cm]{caption}
\usepackage{makecell}
\usepackage{multirow}
\usepackage{float}
\setcellgapes{4pt}

\algrenewcommand\algorithmicrequire{\textbf{Precondition:}} 
\algrenewcommand\algorithmicensure{\textbf{Postcondition:}}
\algrenewcommand\alglinenumber[1]{\footnotesize #1} 
\algnewcommand{\IIf}[1]{\State\algorithmicif\ #1\ \algorithmicthen} 
\algnewcommand{\EndIIf}{\unskip\ \algorithmicend\ \algorithmicif}

\usepackage{graphicx}
\usepackage{setspace}
\usepackage{program}
\usepackage{subfig}
\usepackage{color}

\renewcommand{\quad}{$~~~\;\;\;$}

\newtheorem{proposition}{Proposition}

\newcommand{\diag}{\mbox{\rm diag}}

\newcommand{\transp}{{^{\rm T}}}
\newcommand{\ri}{{\rm ri\,}}

\renewcommand{\R}{\mathbb{R}}

\newcommand{\matr}[1]{\begin{bmatrix} #1 \end{bmatrix}}    
\def\transp{^{\rm T}}

\newcommand{\ip}[2]{\left\langle #1 , #2 \right\rangle}    
\def\int{\mathrm{int}}

\providecommand{\newoperator}[3]{%
\newcommand*{#1}{\mathop{#2}#3}}
\newoperator{\argmax}{\mathsf{argmax}}{}
\newoperator{\argmin}{\mathsf{argmin}}{}
\newcommand{\1}{\mathbf{1}}

\allowdisplaybreaks
\usepackage[margin=2.2cm]{geometry}
\allowdisplaybreaks
\author{Javier Pe\~na\thanks{Tepper School of Business,
Carnegie Mellon University, USA, {\tt jfp@andrew.cmu.edu}}
\and  Negar Soheili\thanks{College of Business Administration,  University of Illinois at Chicago, USA, {\tt nazad@uic.edu }}
}
\title{Computational performance of a projection and rescaling algorithm} 

\begin{document}
\maketitle

\abstract{
This paper documents a computational implementation of a {\em projection and rescaling algorithm} for finding most interior solutions to the pair of feasibility problems
\[
\text{find}  \; x\in L\cap\R^n_{+} \;\;\;\; \text{ and } \; \;\;\;\;
\text{find} \; \hat x\in L^\perp\cap\R^n_{+},
\]
where $L$ denotes a linear subspace in $\R^n$ and $L^\perp$ denotes its orthogonal complement.  The projection and rescaling algorithm is a recently developed method that combines a {\em basic procedure} involving only low-cost operations with a periodic {\em rescaling step.}
We give a full description of a MATLAB implementation of this algorithm and present multiple sets of numerical experiments on synthetic problem instances with varied levels of conditioning.  Our computational experiments provide promising evidence of the effectiveness of the projection and rescaling algorithm.  

Our MATLAB code is publicly available.  Furthermore, the simplicity of the algorithm makes a computational implementation in other environments completely straightforward.

}

\section{Introduction}

The projection and rescaling algorithm~\cite{PenaS16} is a recent polynomial-time algorithm designed for solving the polyhedral feasibility problem 
\begin{equation}\label{primal}
\text{find} \; x\in L\cap\R^n_{++},
\end{equation} 
where $L$ denotes a linear subspace in $\R^n$.  

The projection and rescaling algorithm works by combining two building blocks, namely a {\em basic procedure} and a {\em rescaling step} as follows.  Let $P_L:\R^n \rightarrow L$ denote the orthogonal projection onto $L$.  Within a bounded number of low-cost iterations, the basic procedure finds $z\in\R^n_{++}$ such that either 
\begin{equation}\label{SolvedCondition}
P_Lz \in \R^n_{++}
\end{equation}
or 
\begin{equation}\label{rescalingCondition}
\|(P_Lz)^+\|_1 \leq  {  \dfrac{1}{2} \|z\|_\infty,}
\end{equation}
where $(P_Lz)^+ = \max\{0,P_Lz\}$. If \eqref{SolvedCondition} holds, then $x = P_Lz \in L \cap \R^n_{++}$ is a solution to the original problem~\eqref{primal}.  On the other hand, if \eqref{rescalingCondition} holds {   and $z_i = \|z\|_\infty$ then for every feasible solution $x$ to~\eqref{primal} we have 
\[x_i  \le \frac{1}{\|z\|_\infty}\ip{z}{x} = \frac{1}{\|z\|_\infty}\ip{z}{P_Lx} = \frac{1}{\|z\|_\infty}\ip{P_L z}{x} \le \frac{1}{\|z\|_\infty}\|(P_L z)_+\|_1 \cdot\|x\|_\infty \leq \frac{1}{2} \|x\|_\infty.\] 
In other words, if \eqref{rescalingCondition} holds and $z_i = \|z\|_\infty$ then all solutions $x$ to~\eqref{primal} have small $i$-th component.} The rescaling step takes $D:= I + e_ie_i\transp$ and transforms problem~\eqref{primal} into the following equivalent rescaled problem:
\begin{equation}\label{rescaled}
\text{find} \; x\in D(L)\cap\R^n_{++}.
\end{equation} 
{   Observe that the solutions to the rescaled problem \eqref{rescaled} are in one-to-one correspondence with the solutions to~\eqref{primal} via doubling of the $i$-th component.}

As it is easy to see and detailed in~\cite{PenaS16}, when $D$ is as above,  the rescaled problem~\eqref{rescaled} is better conditioned than~\eqref{primal} in the following sense.  If $L \cap \R^n_{++}\ne \emptyset$ then
$
\delta(D(L)\cap \R^n_{++}) = 2 \delta(L \cap \R^n_{++}) 
$
where $\delta(L\cap\R^n_{++})$ is the following {\em condition measure} of the problem~\eqref{primal}: 
\begin{equation}\label{eq:ConditionMeasure}
\delta(L\cap\R^n_{++}) := \max\left\{\prod_{j=1}^n x_j: x\in L\cap\R^n_{++}, \|x\|_\infty = 1\right\}. 
\end{equation}
By convention $\delta(L \cap \R^n_{++}) = -\infty$ when $L \cap \R^n_{++} = \emptyset$.  Observe that $\delta(L \cap \R^n_{++}) \le 1$ is a measure of the {\em most interior} solution to~\eqref{primal}.  As detailed in~\cite{PenaS16}, it  follows that when $L\cap\R^n_{++}\ne \emptyset$, the  projection and rescaling algorithm finds a solution to~\eqref{primal} in at most $\log_2(1/\delta(L\cap \R^n_{++}))$ rounds of basic procedure and rescaling step.  Furthermore, each round of basic procedure and rescaling step requires a number of elementary operations that is bounded by a low-degree polynomial (quadratic or cubic) on $n$. 

The above projection and rescaling algorithm was originally proposed by Chubanov~\cite{Chub12,Chub15} and is in the same spirit as other rescaling methods in~\cite{BellFV09,Freund85,PenaS13}.  In addition to~\cite{PenaS16}, a number of articles~\cite{DaduVZ16,DaduVZ17,HobeR17,KitaT18,LiRT15,LourKMT16,Roos18} have proposed new algorithmic developments by extending the  projection and rescaling  templates introduced in~\cite{BellFV09,Chub12,Chub15,PenaS13}.
However, despite their interesting theoretical guarantees, there has been limited work on the computational effectiveness of the projection and rescaling algorithm as well as other methods based on rescaling.  As far as we know, only the articles by Li et al.~\cite{LiRT15} and by Roos~\cite{Roos18} report numerical results on implementations of some variants of Chubanov's projection and rescaling algorithm. 
 
This paper documents a MATLAB implementation of an enhanced version of the projection and rescaling algorithm from~\cite{PenaS16}.  Our work differs from~\cite{LiRT15,Roos18} in several ways.  Unlike the  algorithms in~\cite{LiRT15,Roos18}, our main algorithm solves both feasibility problems $L\cap \R^n_+$ and  $L^\perp\cap \R^n_+$ in a symmetric fashion. We also perform and report a significantly larger set of computational experiments in higher level of detail.  We compare, via numerous experiments, the performance of several possible schemes for the basic procedure.  We provide full descriptions of the algorithms that we implement.  The MATLAB code for our implementation is publicly available at the following website:

\begin{center} {\tt http://www.andrew.cmu.edu/user/jfp/epra.html} \end{center} 

All of the numerical experiments reported in this paper can be easily replicated and verified via the above code.  Furthermore,  since our MATLAB code is a verbatim implementation of the algorithms described in the sequel, it is straightforward to replicate our implementation in other numerical computing environments such as R, python, or Julia.

Algorithm~\ref{algo.MPRA}, the main algorithm in our implementation, incorporates the following enhancements to the original Projection and Rescaling Algorithm in~\cite{PenaS16}:

\begin{enumerate}
\item Let $L^\perp$ denote the orthogonal complement of $L$. Algorithm~\ref{algo.MPRA} finds {\em most interior} solutions to the problems
\begin{equation}\label{primal.again}
\text{ find} \;x \in L \cap \R^n_{+},
\end{equation}
and
\begin{equation}\label{dual}
\text{ find} \;\hat x \in L^\perp \cap \R^n_{+}.
\end{equation}
That is,  Algorithm~\ref{algo.MPRA} terminates with points in the relative interiors of $L \cap \R^n_{+}$ and $L^\perp \cap \R^n_{+}$.    In particular, if~\eqref{primal.again} is strictly feasible then Algorithm~\ref{algo.MPRA} finds a point in $L \cap \R^n_{++}.$  Likewise, if~\eqref{dual} is strictly feasible then Algorithm~\ref{algo.MPRA} finds a point in $L^\perp \cap \R^n_{++}.$

Unlike the  Projection and Rescaling Algorithm in~\cite{PenaS16} and the algorithms in~\cite{Chub12,Chub15,LiRT15}, Algorithm~\ref{algo.MPRA} requires no prior feasibility assumptions about~\eqref{primal.again} or~\eqref{dual}.



\item We enforce an upper bound on the size of the entries of the diagonal rescaling matrices maintained throughout Algorithm~\ref{algo.MPRA}.  The upper bound achieves two major goals.  First, it prevents numerical overflow.  Second, it yields a natural criteria to determine when the algorithm has found points in the relative interiors of $L \cap \R^n_{+}$ and $L^\perp \cap \R^n_{+}$.

\item In contrast to the rescaling step in the original Projection and Rescaling Algorithm that rescales $L$ only in one direction at each round, the rescaling step in Algorithm~\ref{algo.MPRA} performs a more aggressive rescaling along {\em all} directions that can improve the conditioning of the problem. This enhancement is fairly similar to a multiple direction rescaling step introduced by Louren{\c{c}}o et al~\cite{LourKMT16}.  It is also similar in spirit to an idea proposed by Roos~\cite{Roos18} to obtain sharper rescaling via a modified basic procedure.

\end{enumerate}

The first two enhancements above enable Algorithm~\ref{algo.MPRA} to apply without the kind of feasibility assumption required by the original Projection and Rescaling Algorithm, namely that $L \cap \R^n_{++} \ne \emptyset$ or $L^\perp \cap \R^n_{++} \ne \emptyset$ and without concerns about numerical overflow due to excessively large rescaling.  On the flip side, the correct termination of Algorithm~\ref{algo.MPRA} readily follows from the results in~\cite{PenaS16} only when one of conditions $L \cap \R^n_{++} \ne \emptyset$ or $L^\perp \cap \R^n_{++} \ne \emptyset$ holds and $U$ is sufficiently large.  Although our numerical experiments demonstrate that Algorithm~\ref{algo.MPRA} correctly terminates in the  majority of the cases, a formal proof of correct termination in the case when  both $L \cap \R^n_{++} =\emptyset$ and $ L^\perp \cap \R^n_{++} = \emptyset$ is not known yet.  The natural conjecture is that Algorithm~\ref{algo.MPRA} correctly terminates when $U$ is sufficiently large.  We will tackle this interesting theoretical question in some future work.

The basic procedure is the main building block of Algorithm~\ref{algo.MPRA}.  We make a separate comparison of the performance of the following four different schemes for the basic procedure proposed in~\cite{PenaS16}: perceptron, von Neumann, von Neumann with away-steps, and smooth perceptron schemes.  
These four schemes are described in Algorithm~\ref{alg:perceptron} through Algorithm~\ref{alg:smooth} below.   According to the theoretical results established in~\cite{PenaS16}, the first three of these schemes have similar convergence rates while Algorithm~\ref{alg:smooth} (the smooth perceptron scheme) has a faster convergence rate but each main iteration of this scheme is computationally more expensive. Section~\ref{sec:experiments.basic} describes various numerical experiments that compare the performance of the four schemes.  The experiments consistently demonstrate that indeed Algorithm~\ref{alg:smooth} has the best performance by a wide margin.  Therefore, we use Algorithm~\ref{alg:smooth} as the basic procedure within 
Algorithm~\ref{algo.MPRA}.  Section~\ref{sec:experiments.mpra} describes results on various numerical experiments that test the performance of  
Algorithm~\ref{algo.MPRA}.  Our results demonstrate the significant advantage of using aggressive rescaling {   and confirm a similar observation by Chubanov~\cite[Section 4.2]{Chub15}}.  They also provide promising evidence that Algorithm~\ref{algo.MPRA} can solve instances of moderate size.

The two main sections of the paper are organized as follows. In Section~\ref{sec:rescaling} we describe our enhanced version of the projection and rescaling algorithm.  This section also recalls four different schemes for the basic procedure proposed in~\cite{PenaS16}.  In Section~\ref{sec:experiments} we present several sets of numerical experiments.  To generate interesting problem instances, we devise a procedure to generate problem instances with arbitrary level of conditioning.  We perform several numerical experiments to compare the different schemes for the basic procedure.  We also perform a number of experiments to test the effectiveness of the enhanced projection and rescaling algorithm.

\section{Enhanced projection and rescaling algorithm}\label{sec:rescaling}

\subsection{Main algorithm}
Algorithm~\ref{algo.MPRA} below describes an enhanced version of the Projection and Rescaling Algorithm from~\cite{PenaS16}.  The algorithm relies on the following characterization of the relative interiors of $L\cap \R^n_+$ and $L^\perp\cap \R^n_+$ for a linear subspace $L \subseteq \R^n$.  The characterization in Proposition~\ref{prop.partition} is a consequence of the classical Goldman-Tucker partition theorem as detailed in~\cite{CheuCP03}.

\begin{proposition}\label{prop.partition} Let $L \subseteq \R^n$ be a linear subspace.  Then there exists a unique partion $B \cup N = \{1,\ldots,n\}$ such that
\[
\ri(L\cap \R^n_+) = \{x\in L \cap \R^n_+: x_i > 0 \text{ for all } i \in B\},
\]
and
\[
\ri(L^\perp\cap \R^n_+) = \{\hat x \in L^\perp \cap \R^n_+: \hat x_i > 0 \text{ for all } i \in N\}.
\]
In particular, $x\in \ri(L\cap \R^n_+)$ and $\hat x \in \ri(L^\perp\cap \R^n_+)$ if and only if $x\in L,\; \hat x\in L^\perp$, and
\begin{equation}\label{eq.relint}
x_B > 0, \, x_N = 0 \; \text{ and } \; \hat x_N > 0, \, \hat x_B = 0.
\end{equation}
\end{proposition}

Observe that $B = \{1,\dots,n\}$ and $N = \emptyset$ when $L\cap \R^n_{++} \ne \emptyset$.  Similarly,  $B = \emptyset$ and $N = \{1,\dots,n\}$  when   $L^\perp\cap \R^n_{++} \ne \emptyset$.  In both of these cases we shall say that the partition $(B,N)$ is {\em trivial}.  We shall say that the partition $(B,N)$ is {\em non-trivial} otherwise, that is, when $B\ne \emptyset$ and $N \ne \emptyset$.

\medskip

Each main iteration of Algorithm~\ref{algo.MPRA} applies the following steps.  First, apply the basic procedure to $D(L) \cap \R^n_+$ and $\hat D(L^\perp) \cap \R^n_+$ for some diagonal rescaling matrices $D$ and $\hat D$.  Next, identify a potential partition $(B,N)$ and terminate if the basic procedures yield $x\in L$ and $\hat x \in L^\perp$ satisfying~\eqref{eq.relint}.  Otherwise, update the rescaling matrices $D$ and $\hat D$ and proceed to the next main iteration: Apply the basic procedure to $D(L) \cap \R^n_+$ and $\hat D(L^\perp) \cap \R^n_+$, etc. 

To prevent numerical overflow, Algorithm~\ref{algo.MPRA} caps the entries of the rescaling matrices $D$ and $\hat D$ by some pre-specified upper bound $U$.  This upper bound naturally determines a numerical threshold to verify if the algorithm has found solutions in the relative interiors of $L \cap \R^n_+$ and $L^\perp \cap \R^n_+$.  More precisely, Algorithm~\ref{algo.MPRA} will terminate with points $x \in L$ and $\hat x \in L^\perp$ satisfying the following approximation of~\eqref{eq.relint}:
\[
x_B > 0, \; \|x_N\|_\infty \le \frac{1}{U} \|x\|_\infty \; \text{ and } \;
\hat x_N > 0, \; \|\hat x_B\|_\infty \le \frac{1}{U} \|\hat x\|_\infty. 
\]

{\centering\begin{minipage}{\linewidth}
\begin{algorithm}[H]
  \caption{Enhanced Projection and Rescaling Algorithm (EPRA)
    \label{algo.MPRA}}
  \begin{algorithmic}[1]
    \State  ({\bf Initialization}) 
    \Statex Let $D := I$ and $\hat D := I$. 
       \Statex Let $P := P_L$ and $\hat P := P_{L^\perp}$.  
\Statex Let $U >0$ be a pre-specified upper bound on the rescaling matrices $D$ and $\hat D$.
	\State ({\bf Basic Procedure})
	\Statex 
	Find  $z \gneqq 0$ such that either $Pz > 0$ or $\|(Pz)^+\|_1 \le \frac{1}{2} \|z\|_\infty$.
	\Statex
	Find  $\hat z \gneqq 0$ such that either $\hat P\hat z > 0$ or $\|(\hat P\hat z)^+\|_1 \le \frac{1}{2} \|\hat z\|_\infty$.

\State ({\bf Partition identification})
\Statex Let $x:=D^{-1}Pz$ and $\hat x := \hat D^{-1}\hat P\hat z$.
\Statex Let $B:=\{i: \vert\hat x_i\vert < \frac{1}{U}\|\hat x\|_\infty\}$ and $N:=\{i:\vert x_i \vert< \frac{1}{U}\|x\|_\infty\}$.
	\IIf{$(B,N)$ partitions $\{1,\dots,n\}$}
	\Statex
HALT and \Return $x \in \ri(L \cap \R^n_{+}), \; \hat x \in \ri(L^\perp \cap \R^n_{+})$ \EndIIf  
  	\State ({\bf Rescaling step}) 
	\Statex Let $e:= \left(z/\|( P z)^+\|_1 -1\right)^+$, $D:=\min\left((I+\diag(e))D,U\right)$ and $P:=P_{D(L)}$. 
	\Statex Let $\hat e:= \left(\hat z/\|(\hat P\hat z)^+\|_1 -1\right)^+$, $\hat D:=\min\left((I+\diag(\hat e))\hat D,U\right)$ and $\hat P:=P_{\hat D(L^\perp)}$. 
	\Statex Go back to step 2.
	\end{algorithmic}
\end{algorithm}
\end{minipage}}

\subsection{Basic procedure}\label{basicprocedures}

Let $P:\R^n \rightarrow \R^n$ be the orthogonal projection onto a linear subspace of $\R^n$ and $\epsilon \in (0,1)$.  The goal of the basic procedure is to find a non-zero $z\in \R^n_+$ such that either $Pz > 0$ or $\|(Pz)^+\|_1 \le \epsilon \|z\|_\infty$.  We choose $\epsilon = 1/2$ when the basic procedure is used within Algorithm~\ref{algo.MPRA}.
We next recall the four schemes for the basic procedure proposed in~\cite{PenaS16}.  Algorithm~\ref{alg:perceptron} describes the simplest of these schemes, namely the perceptron scheme.  In the algorithms below $\Delta_{n-1}$ denote the standard simplex in $\R^n$ ,that is, \[
\Delta_{n-1} = \{x\in \R^n_+: \|x\|_1 = 1\}.
\]

{\centering\begin{minipage}{\linewidth}
\begin{algorithm}[H]
  \caption{Perceptron Scheme
    \label{alg:perceptron}}
  \begin{algorithmic}[1]
  \State Pick $z_0\in \Delta_{n-1}$, and $t:=0.$
   \While {$Pz_t \ngtr 0$ and $\|(Pz_t)^+\|_1 > \epsilon \|z_t\|_\infty$}
\Statex \quad Pick $u \in \Delta_{n-1}$ such that $\ip{u}{Pz_t} \le 0$.
\Statex \quad  Let $z_{t+1}:=\left(1-\frac{1}{t+1}\right) z_t  + \frac{1}{t+1} u.$
\Statex \quad $t :=t+1.$
\EndWhile
\end{algorithmic}
\end{algorithm}
\end{minipage}}

\bigskip

Algorithm~\ref{alg:vonNeumann} describes the second basic procedure scheme, namely the von Neumann scheme.  This scheme is a greedy variant of the perceptron scheme. This algorithm relies on the following mapping
\[
u(v) := \argmin_{u\in\Delta_{n-1}} \ip{u}{v}.
\]
At each iteration, Algorithm~\ref{alg:vonNeumann} chooses $z_{t+1}$ as the convex combination of $z_t$ and $u(Pz_t)$ that minimizes $\|Pz_{t+1}\|_2$.

{\centering\begin{minipage}{\linewidth}
\begin{algorithm}[H]
  \caption{Von Neumann Scheme
    \label{alg:vonNeumann}}
  \begin{algorithmic}[1]
  \State Pick $z_0\in \Delta_{n-1}$, and $t:=0.$
   \While {$Pz_t \ngtr 0$ and $\|(Pz_t)^+\|_1 > \epsilon\|z_t\|_\infty$}
\Statex \quad Let $u=u(Pz_t)$ and $z_{t+1}:= z_t + \theta(u-z_t)$
 where
\[
\theta_t = \argmin_{\theta\in[0,1]} \|P(z_t + \theta (u-z_t))\|_2^2 = 
\dfrac{\|Pz_t\|_2^2 -\ip{u}{Pz_t}}{\|Pz_t\|_2^2 + \|Pu\|_2^2 - 2\ip{u}{Pz_t}}.
\] 
\Statex \quad $t :=t+1.$
\EndWhile

\end{algorithmic}
\end{algorithm}
\end{minipage}}

\bigskip

Algorithm~\ref{alg:vonNeumann.away} describes the third basic procedure scheme, namely the von Neumann with away steps scheme, which in turn is a variant of the von Neumann scheme.  Algorithm~\ref{alg:vonNeumann.away} relies on the following construction. Define the {\em support} of a current iterate $z$ as $S(z):=\{i\in\{1,\ldots,n\}: z_i > 0\}.$  At each main iteration Algorithm~\ref{alg:vonNeumann.away} chooses between two different kinds of steps: {\em regular} steps as in Algorithm~\ref{alg:vonNeumann} and {\em away steps} that  decrease the weight on a component of $z$  belonging to $S(z)$.  The away steps are computed via the mapping 
\[v(z):=\argmax_{v\in\Delta_{n-1}\atop S(v) \subseteq S(z)} \ip{v}{Pz}.
\]

{\centering\begin{minipage}{\linewidth}
\begin{algorithm}[H]
  \caption{Von Neumann with Away Steps Scheme
    \label{alg:vonNeumann.away}}
  \begin{algorithmic}[1]
  \State Pick $z_0\in \Delta_{n-1}$, and $t:=0.$
   \While {$Pz_t \ngtr 0$ and $\|(Pz_t)^+\|_1 > \epsilon \|z_t\|_\infty$}
\Statex  \quad Let $u = u(Pz_t)$ and  $v = v(z_t)$.
\Statex  \quad  {\bf if} $\|Pz_t\|^2 - \ip{u}{Pz_t} > \ip{v}{Pz_t} -  \|Pz_t\|^2$ {\bf then} (regular step)

\quad $a:= u-z_t; \; \theta_{\max} = 1$,

\Statex \quad {\bf else} (away step)

\quad $a:= z_t-v; \; \theta_{\max} = \frac{\ip{v}{z}}{1-\ip{v}{z}}$.

\Statex \quad {\bf endif}
\Statex \quad Let $z_{t+1}:= z_t + \theta a$ where
\[
\theta = \argmin_{\theta\in[0,\theta_{\max}]} \|P(z_t + \theta a)\|^2 
=
\min\left\{ \theta_{\max} , -\dfrac{\ip{z_t}{Pa}}{\|Pa\|^2}\right\} 
\]
\Statex  \quad $t :=t+1$.
\EndWhile

\end{algorithmic}
\end{algorithm}
\end{minipage}}
\bigskip

Algorithm~\ref{alg:smooth} describes the fourth basic procedure scheme, namely the smooth perceptron scheme, which in turn is a variant of the the perceptron scheme that relies on the following smooth version of the mapping $u(\cdot)$.  Let $\bar u \in \Delta_{n-1}$ be fixed.  For $\mu > 0$ let
\[
u_\mu(v) := \argmin_{u\in\Delta_{n-1}}\left\{\ip{u}{v} + \frac{\mu}{2}\|u-\bar u\|^2\right\}.
\]

{\centering\begin{minipage}{\linewidth}
\begin{algorithm}[H]
  \caption{Smooth Perceptron Scheme
    \label{alg:smooth}}
  \begin{algorithmic}[1]
  \State Let $u_0 := \bar u$; $\mu_0 = 2$; $z_0:=u_{\mu_0}(Pu_0);$ and $t:=0$.
 \While {$Pz_t \ngtr 0$ and $\|(Pz_t)^+\|_1 > \epsilon\|z_t\|_\infty$}
\Statex  \quad $\theta_t:=\frac{2}{t+3}$ 
\Statex   \quad $ u_{t+1} :=(1-\theta_t)u_t + \theta_t z_t  + \theta_t^2 u_{\mu_t}(Pu_t)$
\Statex  \quad $\mu_{t+1} := (1-\theta_t)\mu_t$
\Statex \quad  $z_{t+1} := (1-\theta_t)z_t + \theta_t u_{\mu_{t+1}}(Pu_{t+1})$
\Statex  \quad $t :=t+1$.
\EndWhile

\end{algorithmic}
\end{algorithm}
\end{minipage}}

\section{Numerical experiments}\label{sec:experiments}

This section describes various sets of numerical experiments that test the Enhanced Projection and Rescaling Algorithm described as Algorithm~\ref{algo.MPRA} above.   We also performed numerical experiments to compare the four schemes for the basic procedure, namely Algorithm~\ref{alg:perceptron} through Algorithm~\ref{alg:smooth} on suitably generated instances.  

\subsection{Schemes to construct challenging instances}
\label{sec.construct}
We should note that except for the case when the dimension of the subspace $L$ is about half the dimension of the ambient space $\R^n$, a naive approach to generate random instances yields results of limited interest.  
More precisely, suppose $L \subseteq \R^n$ is a random subspace generated via $L = \ker(A)$ where the entries  of $A\in\R^{m\times n}$ are independently drawn from a standard normal distribution.  From a classical result on coverage processes by Wendel~\cite[Equation (1)]{Wend62} it follows that 
\begin{equation}\label{eq.prob}
\mathbb{P}(L^\perp\cap\R^n_{++} \ne \emptyset) = 2^{1-n}\sum_{k=0}^{m-1} {{n-1}\choose k}\;\;\text{ and }\;\;\mathbb{P}(L\cap\R^n_{++} \ne \emptyset) = 2^{1-n}\sum_{k=m}^{n-1} {{n-1}\choose k}.
\end{equation}
In particular,~\eqref{eq.prob} implies that if $n$ is even and $\dim(L) = n-m = n/2$ then $L\cap \R^n_{++} \ne \emptyset$ with probability 0.5.  Furthermore,~\eqref{eq.prob} implies that if $\dim(L) = n-m \gg n/2$ then $L\cap \R^n_{++} \ne \emptyset$ with high probability.  Similarly,~\eqref{eq.prob} implies that if $\dim(L) = n-m \ll n/2$ then    $L^\perp\cap \R^n_{++} \ne \emptyset$ with high probability.  The identity~\eqref{eq.prob} also suggests that when $L\subseteq \R^n$ is a random subspace and $\dim(L)$ is far enough from $n/2$ then with high probability $\max\{\delta(L\cap \R^n_{++}),\delta(L^\perp\cap \R^n_{++})\}$ is bounded away from zero as there is extra room for either $L$ or $L^\perp$ to cut deep inside $\R^n_{++}$.  The latter fact can be rigorously stated and justified, albeit in somewhat technical terms, by using the machinery on coverage processes and probabilistic analysis of condition numbers developed by B\"urgisser et al~\cite{BurgCL10}.  Our numerical experiments confirm that indeed most random instances $L$ with either $\dim(L)\gg n/2$ or $\dim(L)\ll n/2$ are easily solvable without rescaling (see Table~\ref{table.interior.naive} in Section~\ref{sec:experiments}).  Therefore such random instances are 
not particularly interesting.

\medskip

We next describe schemes to generate collections of more interesting and  challenging instances.    First, we describe how to generate random subspaces $L\subseteq \R^n$ such that $L\cap \R^n_{++}\ne \emptyset$ with a {\em controlled} condition measure $\delta(L\cap \R^n_{++})$.
We subsequently describe how to generate random subspaces $L\subseteq \R^n$ such that both $L\cap \R^n_{+}$ and $L^\perp \cap \R^n_{+}$ have non-trivial relative interiors.

\begin{proposition}\label{prop.control.cn}
Let $\bar x\in\R^n_{++}$ and $\bar u\in \R^n_+$ be such that $\|\bar x\|_\infty = 1$, $\|\bar u\|_1 = n$ and $\bar u_j = 0$ whenever $\bar x_j < 1$ for $j=1,\dots,n$. 
Let $A = \matr{a_1 & \cdots & a_m}\transp \in \R^{m\times n}$ be such that $a_1 =  \bar u-\bar X^{-1}\1$ and $\ip{a_j}{\bar x} = 0$ for $j=2,\dots,m$ where $\bar X\in \R^{n\times n}$ is a diagonal matrix with elements of $\bar x$ spread across the diagonal and $\1\in\R^n$ is the vector of ones. Then for $L = \ker(A) := \{x\in \R^n : Ax = 0\}$ we have 
\[
\bar x = \argmax\left\{\prod_{j=1}^n x_j: x\in L\cap\R^n_{++}, \|x\|_\infty = 1\right\}.
\]
In particular, $L\cap\R^n_{++} \ne \emptyset$ and $\delta(L\cap\R^n_{++}) = \prod_{j=1}^n \bar x_j.$
\end{proposition}
\begin{proof} It suffices to show that
\begin{align}\label{deltaInfinityrelax}
\bar x &= \argmax_x\left\{\ln\left(\prod_{i=1}^n x_i\right): x\in L \cap \R^n_{++}, \|x\|_\infty = 1\right\} \notag\\
& =\argmax_x\left\{\ln\left(\prod_{i=1}^n x_i\right): x\in L \cap \R^n_{++}, \|x\|_\infty \leq 1\right\}. 
\end{align}
The conditions on the rows of $A$ readily ensure that $\bar x \in  L \cap \R^n_{++}$.  Thus $\bar x$ is a feasible solution to~\eqref{deltaInfinityrelax}.  Since $\zeta = \bar X^{-1} \1$ is the gradient of the objective function in~\eqref{deltaInfinityrelax} at $\bar x$, to show that $\bar x$ is optimal it suffices to show that $\ip{\zeta}{x - \bar x} \leq 0$ for any feasible solution to~\eqref{deltaInfinityrelax}.
Take $x\in L\cap\R^n_{++}$ with $\|x\|_\infty \leq 1$. Since $x\in L$, we have $\ip{\bar X^{-1}\1 - \bar u}{x} = \ip{a_1}{x} = 0$. 
Therefore $\ip{\zeta}{x-\bar x} = \ip{\bar X^{-1} \1}{x} - n \le \ip{\bar u}{ x} - \|\bar u\|_1\|x\|_{\infty} \leq 0$.  The last two steps follow  from $\|\bar u \|_1 = n$ and H\"older's inequality respectively.
\end{proof} \qed

Proposition~\ref{prop.control.cn} readily suggests a scheme to generate subspaces $L \subseteq \R^n$ such that the condition measure $\delta(L\cap \R^n_{++})$ is positive but as small as we wish: pick $\bar x\in \R^n_{++}$ with $\|\bar x\|_\infty =1$ and generate $\bar u\in \R^n_+, A \in \R^{m\times n},$ and $L = \ker(A)$ as in Proposition~\ref{prop.control.cn}.  We next explain how Proposition~\ref{prop.control.cn} can be further leveraged to generate $L \subseteq \R^n$ so that both $L\cap \R^n_+$ and $L^\perp\cap \R^n_+$ have non-trivial relative interiors.  Suppose $(B,N)$ is a partition of $\{1,\dots,n\}$ and 
\begin{equation}\label{eq.block.matrix}
 A = \begin{bmatrix}A_{BB} & A_{NB}\\ 0& A_{NN}\end{bmatrix}
\end{equation}
is such that $L_B = \ker(A_{BB}) \subseteq \R^B$ and $L_N = \text{Im}(A_{NN}\transp) \subseteq \R^N$ satisfy $L_B\cap \R^B_{++}\ne\emptyset$ and $
L_N\cap \R^N_{++}\ne\emptyset$. If $A_{NN}$ is full row-rank then it readily follows that  the subspaces $L = \ker(A)$ and $L^\perp = \text{Im}(A\transp)$ satisfy
\[
\ri(L\cap \R^n_+) = \{x\in L\cap \R^n_+: x_i >0 \text{ for all } i \in B\} 
\]
and
\[
\ri(L^\perp\cap \R^n_+) = \{\hat x\in L^\perp\cap \R^n_+: \hat x_i >0 \text{ for all } i \in N\}. 
\]

Hence we can generate subspaces $L \subseteq \R^n$ such that both $L\cap \R^n_+$ and $L^\perp\cap \R^n_+$ have non-trivial relative interiors by proceeding as follows.  First, choose a partition $(B,N)$ of $\{1,\dots,n\}.$ Next, use the construction suggested by Proposition~\ref{prop.control.cn} to generate full row-rank matrices $A_{BB},A_{NN}$ such that $\ker(A_{BB})\cap \R^B_{++}\ne\emptyset$ and $
\text{Im}(A_{NN}\transp)\cap \R^N_{++}\ne\emptyset$.  Finally let $L = \ker(A)$ where $A$ is of the form~\eqref{eq.block.matrix} for some $A_{NB}$ of appropriate size.

\subsection{Comparison of basic procedure schemes}\label{sec:experiments.basic}

The computational experiments summarized in this section compare the performance of the four schemes for the basic procedure, namely Algorithm~\ref{alg:perceptron} through Algorithm~\ref{alg:smooth}. We implemented these algorithms in MATLAB and ran them on collections of instances defined by $L
 = \text{ker}(A)$, for $A \in \R^{m\times n}$. We used the QR-factorization to obtain the orthogonal projection mappings $P = P_L$ and $\hat P = P_{L^\perp}$. 

We performed two main sets of experiments.  The first set of experiments contains instances $L=\ker(A)$ where the entries of $A \in \R^{m\times n}$ are independently drawn from a standard normal distribution and $m = n/2$ for $n=200, 500, 1000, 2000.$   When $m$ significantly differs from $n/2$,  random instances generated in this way are uninteresting as they can easily be solved by any of the four schemes.
The second set of experiments contains more challenging instances $L=\ker(A)$ where $A\in \R^{m\times n}$ is generated via the procedure suggested by Proposition~\ref{prop.control.cn} for $n=1000$, $m = 100, 200, 800, 900$.   More precisely, we generated $\bar x \in \R^n_{++}$ as follows.  First, we set a random chunk of its entries uniformly at random between 0 and 0.001. Second, we set remaining entries uniformly at random distributed between 0 and 1.  Third, we scaled the entries of $\bar x$ to obtain $\|\bar x\|_\infty = 1$.  Once we generated $\bar x$ in this fashion, we generated $A \in \R^{m\times n}$ as in Proposition~\ref{prop.control.cn}. 

Table~\ref{summaryTableeps1} through Table~\ref{summaryTableeps4.control} 
summarize the results on various sets of experiments.  Each row corresponds to a set of 1000 instances.  To keep the number of iterations and CPU time manageable, we enforced an upper bound of 10000 iterations for all four schemes.  The first two columns in each table indicate the  size of $A \in \R^{m\times n}$.  The other columns display three numbers for each of the four schemes: the average number of iterations, the average CPU time, and the success rate on the batch of 1000 instances of size $m$ by $n$.  The success rate is the proportion of instances on which the scheme terminates normally before reaching the upper bound of 10000 iterations.

Table~\ref{summaryTableeps1} and Table~\ref{summaryTableeps4} display the results for the first set of experiments when $m=n/2$ and $A \in \R^{m\times n}$ is randomly generated without any control on the conditioning of $L\cap\R^n_{++}$.  Table~\ref{summaryTableeps1.control} and Table~\ref{summaryTableeps4.control}  display similar summaries for the second set of experiments where we generate $A\in \R^{m\times n}$ so that $L\cap\R^n_{++}$ has a controlled condition measure via the procedure suggested by Proposition~\ref{prop.control.cn}.  The tables summarize results for two values of  $\epsilon$: $\epsilon = 10^{-1}$ (large), and $\epsilon = 10^{-4}$ (small).

\begin{table}[H]
\centering
\caption{Results for naive random instances, large $\epsilon$ ($\epsilon = 10^{-1}$), {   and 10000 iteration limit}}
{\small
\begin{tabular}{|c|c|c|c|c|c|c|c|c|c|c|}
\hline
 $m$ & $n$ & perceptron & VN & VNA & smooth  \\
 \hline
100&	200&	(6956.28, 0.27, 0.74) &	(5070.41, 0.26, 0.69) &	(3021.73, 0.23, 0.95) &	{\color{blue}(27.04, 0.03, 1)}\\
250&	500&	(9963.91, 0.85, 0.02) & (9207.1, 0.26, 0.2) & (8737.9, 0.23, 0.38) & {\color{blue} (43.88, 0.13, 1)}\\
500&	1000&	(10000, 8.67, 0)&	(9981.29, 8.84, 0.01)&	(9992.46, 14.3, 0.01)&	{\color{blue}(58.50, 0.42, 1)}\\
1000	&2000&	(10000, 34.72, 0)&	(10000, 35.54, 0)&	(10000, 67.24, 0)&{\color{blue}(80.21, 1.42, 1)}\\
\hline
\end{tabular}\label{summaryTableeps1}
}
\end{table}

\begin{table}[H]
\centering
\caption{Results for naive random instances, small $\epsilon$ ($\epsilon = 10^{-4}$),
{   and 10000 iteration limit}}
{\small
\begin{tabular}{|c|c|c|c|c|c|c|c|c|c|c|}
\hline
 $m$ & $n$ & perceptron & VN & VNA & smooth  \\
 \hline
100&	200&	(8236.2, 0.33, 0.35) &	(5395.1, 0.28, 0.67) &	(5861, 0.45, 0.66) &	{\color{blue}(123.6, 0.14, 1)}\\
250&	500&	(9981.9, 0.94, 0.01) & (9258.1, 0.99, 0.2) & (9518.5, 1.64, 0.16) & {\color{blue} (231.8, 0.64, 1)}\\
500&	1000&	(10000, 8.28, 0)&	(9973.7, 8.39, 0.01)&	(9988.5, 13.57, 0.01)&	{\color{blue}(337.64, 2.15, 1)}\\
1000	&2000&	(10000, 35.61, 0)&	(10000, 36.34, 0)&	(10000, 68.71, 0)&{\color{blue}(465.94, 7.87, 1)}\\
\hline
\end{tabular}\label{summaryTableeps4}
}
\end{table}

\begin{table}[H]
\centering
\caption{{   Results for controlled condition instances, large $\epsilon$ ($\epsilon = 10^{-1}$),
 and 10000 iteration limit}}
\begin{tabular}{|c|c|c|c|c|c|c|c|c|c|}

\hline
 $m$ & $n$ & perceptron & VN & VNA & smooth  \\
 \hline
100&	1000&	(9134.38, 0.88, 0.32) &	(8519.12, 8.48, 0.25) &	(3329.84, 6.27, 1) &	{\color{blue}(130.77, 0.30, 1)}\\
200&	1000&	(9649.15, 8.02, 0.21) & (8645.91, 7.35, 0.26) & (5005.64, 8.12, 0.98) & {\color{blue} (140.21, 0.27, 1)}\\
800&	1000&	(3383.55, 2.86, 0.87)&	(9798.34, 8.33, 0.03)&	(6566.22, 10.61, 0.7)&	{\color{blue}(220.53, 0.42, 1)}\\
900	&1000&	(2156.34, 1.92, 0.99)&	(9842.71, 8.71, 0.03)&	(1429.66, 2.43, 1)&{\color{blue}(198.58, 0.39, 1)}\\
\hline
\end{tabular}\label{summaryTableeps1.control}
\end{table}

\begin{table}[H]
\centering
\caption{{   Results for controlled condition instances, small $\epsilon$ ($\epsilon = 10^{-4}$),  and 10000 iteration limit}}
{\small
\begin{tabular}{|c|c|c|c|c|c|}
\hline
$m$ & $n$ & perceptron & VN & VNA & smooth  \\
 \hline
100&	1000&	(9961.9, 8.26, 0) &	(9926.9, 16.06, 0.01) &	(9934.6, 16.01, 0.01) &	{\color{blue}(9463.5, 17.65, 0.16)}\\
200&	1000&	(9952.6, 8.74, 0.01) & (9942.1, 16.95, 0.01) & (9950.1, 16.95, 0.01) & {\color{blue} (9579.9, 19.16, 0.13)}\\
800&	1000&	(9964.52, 8.94, 0)&	(9961.16, 17.15, 0)& (9967.6, 17.15, 0)&	{\color{blue}(8557.5, 17.05, 0.75)}\\
900	&1000&	(9939.5, 8.80, 0.01)&	(9915.7, 16.94, 0.01)&	(9939.1, 16.94, 0.01)&{\color{blue}(7537.2, 14.85, 0.97)}\\
\hline
\end{tabular}\label{summaryTableeps4.control}
}
\end{table}

\bigskip

When $\epsilon$ is large (Table~\ref{summaryTableeps1} and Table~\ref{summaryTableeps1.control}), the algorithms often stop when the condition $\|(Pz)^+\|_1\le \epsilon {  \|z\|_\infty}$ is satisfied. Not surprisingly, when $\epsilon$ is small (Table~\ref{summaryTableeps4} and Table~\ref{summaryTableeps4.control}), the basic procedures more often stop when $Pz >0$  and require a larger number of iterations and longer CPU time.  Also as expected, when the instances become larger, they become more challenging and more iterations are needed to find a feasible solution.

Our numerical experiments for large $\epsilon$ demonstrate that the smooth perceptron scheme is faster both in number of iterations and in terms of CPU time than any of the other three schemes.  The experiments also suggest that {   when enforcing the 10000 iteration limit} the perceptron, von Neumann, and von Neumann with away steps  are comparable in terms of number of iterations and CPU time.  Given the evidence in favor of the smooth perceptron scheme, we use this method within the Enhanced Projection and Rescaling Algorithm.

We note that the numerical experiments for small $\epsilon$ in Table~\ref{summaryTableeps4.control} confirm that the scheme for generating challenging instances indeed yields instances that are difficult to solve for all schemes and thus provide an interesting testbed for the Enhanced Projection and Rescaling Algorithm.

{  
The low success rates in some of the entries in Table~\ref{summaryTableeps1} through Table~\ref{summaryTableeps4.control} reveal that for many instances the upper limit of 10000 iterations is reached by the perceptron, von Neumann, and von Neumann with away steps schemes.  Thus for additional robustness check, we also performed some extra sets of experiments 
without any limit on the number of iterations.  The results are summarized in Table~\ref{new.table1} and Table~\ref{new.table2}.    We ran fewer instances and used $\epsilon = 10^{-1}$ to keep the experiments manageable.  (Some schemes run for over several million iterations in some instances.)  The last four columns of  Table~\ref{new.table1} and Table~\ref{new.table2} report only the average number of iterations and average CPU times since all instances are run until successful termination without iteration limit.  Each row corresponds to a set of 100 instances except for the last row for $m=1000, \; n= 2000$.  In this case we only ran 20 instances due to time limitations.  Without iteration limit, some of these instances take multiple hours of CPU time.

The results in these two tables further confirm that the smooth perceptron scheme is faster both in number of iterations and in terms of CPU time than any of the other three schemes.  Furthermore, the additional experiments suggest that without iteration limit the von Neumann scheme usually requires the highest number of iterations.}

\begin{table}[H]
\centering
\caption{{  Results for naive random instances, large $\epsilon$ ($\epsilon = 10^{-1}$), and no iteration limit}}
{\small
\begin{tabular}{|c|c|c|c|c|c|c|c|c|c|c|}
\hline
 $m$ & $n$ & perceptron & VN & VNA & smooth  \\
 \hline
100&	200&	(8780.80, 0.31) &	(24453.42, 1.11) &	(3054.96, 0.22) &	{\color{blue}(66.7, 0.008)}\\            
250&	500&	(62807.14, 11.8) & (565958.64, 112.7) & (22853.18, 8.13) & {\color{blue} (122.18, 0.06)}\\
500&	1000&	(267348.4, 227.73)&	(2301999.2, 2017.9)&	(91897.2, 151.3)&	{\color{blue}(177.0,0.34)}\\
1000	&2000&	(856072.0, 2739.47)&	(717508.8, 2331.06)&	(162726.9, 1010.66)&{\color{blue}(226.6, 1.6)}\\
\hline
\end{tabular}\label{new.table1}
}
\end{table}

\begin{table}[H]
\centering
\caption{{  Results for controlled condition instances, large $\epsilon$ ($\epsilon = 10^{-1}$), and no iteration limit}}
\begin{tabular}{|c|c|c|c|c|c|c|c|c|c|}

\hline
 $m$ & $n$ & perceptron & VN & VNA & smooth  \\
 \hline
100	&1000&	(14982.34, 15.64)&	(47942.98, 51.68)&	(3585.23, 6.97 )&{\color{blue}(127.93, 0.29)}\\
200&	1000&	(16037.40, 16.67)&	(57652.07, 62.43)&	(5152.81, 9.87)&	{\color{blue}(146.49, 0.33)}\\
800&	1000&	(9157.77, 8.8) & (892550.84, 905.88) & (7243.26, 13.11) & {\color{blue} (220.8, 0.47)}\\
900&	1000&	(1975.82, 1.93) &	(692402.25, 695.27) &	(1410.76, 2.46) &	{\color{blue}(199.8, 0.43)}\\
\hline
\end{tabular}\label{new.table2}
\end{table}

\subsection{Performance of the Enhanced Projection and Rescaling Algorithm}
\label{sec:experiments.mpra}
This section describes the performance of Algorithm~\ref{algo.MPRA} on two main sets of problem instances.  The first set contains  instances of $L = \ker(A)$ for $A\in \R^{m\times n}$ with $L \cap \R^n_{++} \ne \emptyset$  generated via the approach based on Proposition~\ref{prop.control.cn} as described in Section~\ref{sec.construct}. The second set of instances $L = \ker(A)$ is also generated via a similar approach but ensuring that both $\ri(L \cap \R^n_{+}) \ne \{0\}$ and $\ri
(L^\perp \cap \R^n_{+}) \ne \{0\}$. 
Most of these instances are sufficiently challenging that they cannot be solved by the basic procedure (via the smooth perceptron scheme) without rescaling.  

We ran Algorithm~\ref{algo.MPRA} with $U=10^{10}$ in all of our experiments.  Table~\ref{table.interior} displays the results for the first set of instances $L = \ker(A)$ with $L \cap \R^n_{++} \ne \emptyset$.  Each row corresponds to a set of 500 instances of $A\in \R^{m\times n}$ for  $m$ and $n$ as indicated in the first two columns.  The other three columns display the average number of rescaling iterations, average total number of basic procedure iterations, and average CPU time for each set of 500 instances.  Furthermore, we note that Algorithm~\ref{algo.MPRA} successfully solves all instances, that is, it terminates with a point $x\in L \cap \R^n_{++}$.  It is noteworthy that the number of rescaling iterations ranges from 9 to 15 across instances of different sizes.  To further illustrate this interesting fact, Figure~\ref{fig.rescaling} plots the number of rescaling iterations for some sets of instances.

\begin{table}[H]
\centering
\caption{{  Algorithm~\ref{algo.MPRA} on controlled condition instances with $L\cap \R^n_{++} \ne \emptyset$}}
\begin{tabular}{|c|c|c|c|c|}
\hline
$m$ & $n$& \makecell{Average \# of rescaling\\ iterations}&\makecell{Average total \# of\\ iterations}& \makecell{Average CPU time\\ (in seconds)}\\
\hline
100	  & 200  &	9.51   & 712.38   &	0.076\\
250  & 500  &	11.03 & 1419.98  &0.76\\
100  & 1000 &	7.97   & 1843.74 &	4.48\\
200  & 1000 &	9.00   & 1954.73 &	4.41\\
500  & 1000 &	11.92 & 2487.47 &	5.49\\
800  & 1000 &	13.08 & 4026.00 &	6.69\\
900  & 1000 &	11.63 & 4184.90 &	6.64\\
1000& 2000 &	12.18 & 4318.48 & 35.76\\
\hline
\end{tabular}\label{table.interior}
\end{table}\FloatBarrier

\begin{figure}
\centering
 \includegraphics[width=.7\textwidth] {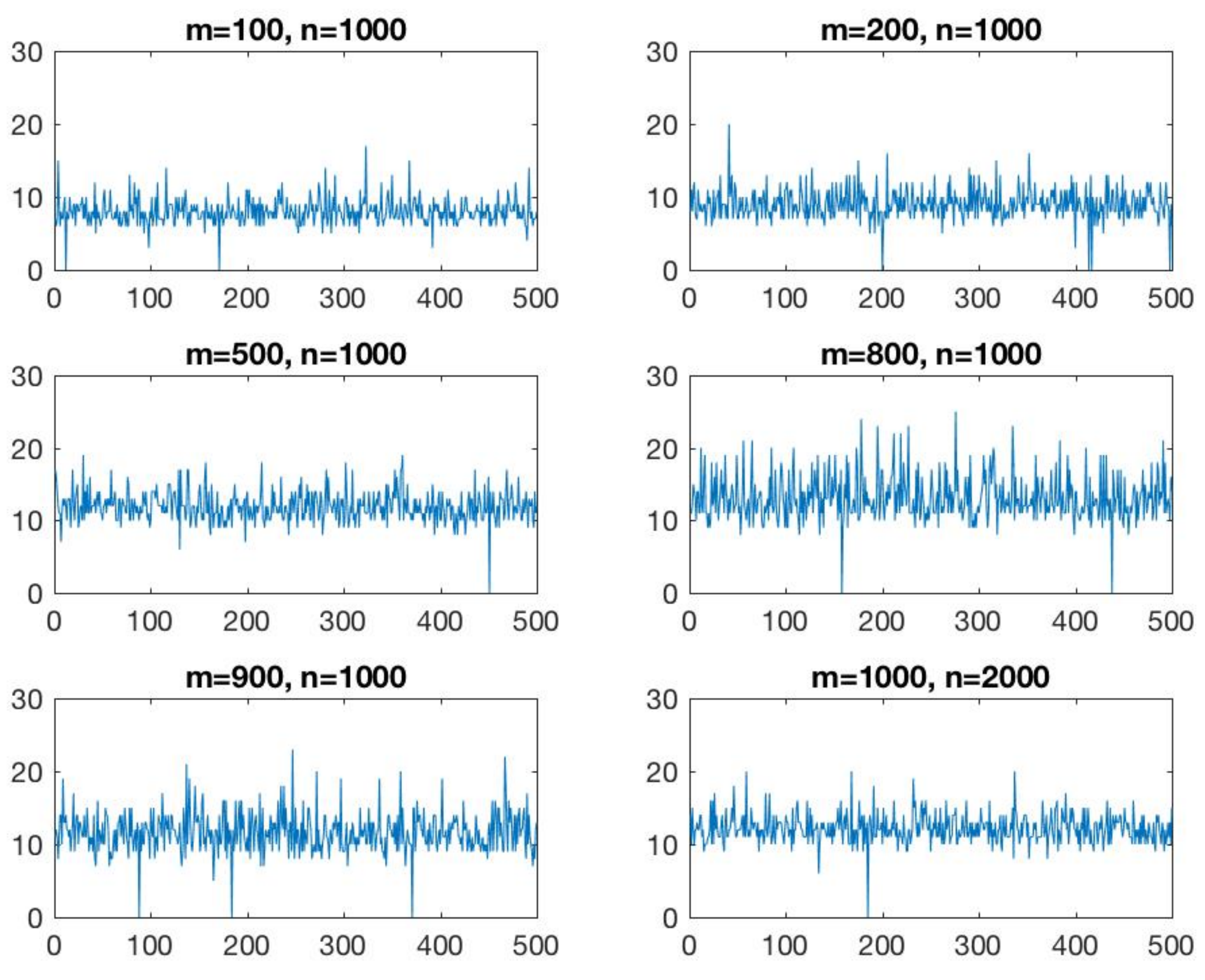}
\caption{Number of rescaling iterations for controlled condition instances $L$ with $L\cap\R^n_{++}\ne \emptyset.$}
\label{fig.rescaling}
\end{figure}

Table~\ref{table.partition} and Figure~\ref{fig.rescaling.part} display similar results for the second set of instances with both $\ri(L \cap \R^n_{+}) \ne \emptyset$ and $\ri(L^\perp \cap \R^n_{+}) \ne \emptyset$.  To accommodate for a wide and flexible range of dimensions of $\ri(L \cap \R^n_{+}) \ne \emptyset$ and $\ri(L^\perp \cap \R^n_{+}) \ne \emptyset$, for each fixed value of $n$ we construct $A\in \R^{m\times n}$ and $L=\ker(A)$ with varying values of $m$.  For this second set of instances we also report the {\em success rate}, that is, the percentage of instances where the the partition $(B,N)$ is correctly identified.  The algorithm succeeds in identifying this partition for most instances.  In the rare cases when this is not the case, failure occurs because either $B$ or $N$ are small and Algorithm~\ref{algo.MPRA} terminates with a point that is either in $L^\perp \cap \R^n_{++}$  or in $L \cap \R^n_{++}$ within roundoff error.
The experiments show that on this second set of instances a higher number of rescaling iterations is usually necessary.  This is somewhat expected as these instances include the extra difficulty of finding a non-trivial partition $(B,N)$ of $\{1,\dots,n\}.$

\begin{table}[H]
\centering
\caption{Algorithm~\ref{algo.MPRA} on  instances with 
$\ri(L\cap \R^n_{+}) \ne \{0\}$ and $\ri(L^\perp\cap \R^n_{+}) \ne \{0\}$}
\begin{tabular}{|c|c|c|c|c|c|}
\hline
$n$ & \makecell{Average\\$m$}& \makecell{Average \# \\ of rescaling\\ iterations}&\makecell{Average total \# \\ of iterations}& \makecell{Average \\ CPU time\\ (in seconds)} & \makecell{Success \\ rate}\\
\hline
100	&51.21&		16.16& 657.01&	0.075 & 0.946\\
200	&100.55&		17.27& 1031.45&	0.12 & 0.974\\
500&	249.64&		17.60 & 1763.59&	1.09& 0.984\\
800&	405.86&		17.72 & 2269.69&	3.52& 0.998\\
1000&	499.72&		17.76& 2625.84&	6.73& 0.994\\
2000&	1006.93&		17.59& 3752.04&	47.07& 0.994\\
\hline
\end{tabular}\label{table.partition}
\end{table}\FloatBarrier

\begin{figure}
\centering
 \includegraphics[width=.7\textwidth] {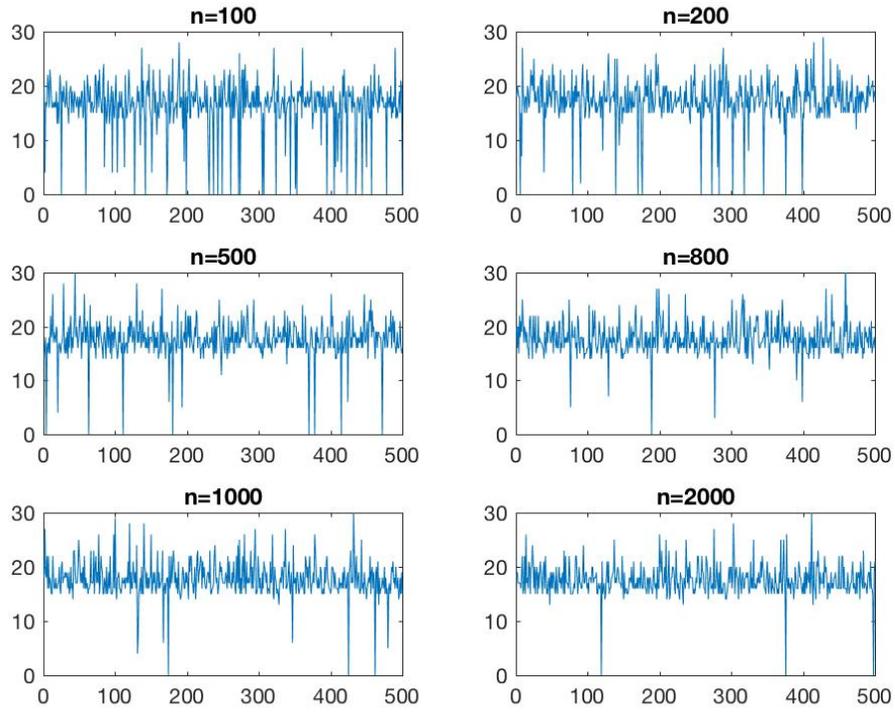}
\caption{Number of rescaling iterations for instances $L$ with $\ri(L\cap\R^n_{+})\ne \{0\}$ and $\ri(L^\perp\cap\R^n_{+})\ne \{0\}.$}
\label{fig.rescaling.part}
\end{figure}

To further illustrate the partition and the solutions  found by Algorithm~\ref{algo.MPRA}, Figure~\ref{fig:partitioninxands1k} and Figure~\ref{fig:partitioninxands2k} plot the coordinates of the points $x$ and $\hat x$ found by  Algorithm~\ref{algo.MPRA} for two representative instances of dimension $n = 1000$ and $n = 2000$ respectively. The two plots in the first row of Figure~\ref{fig:partitioninxands1k} and Figure~\ref{fig:partitioninxands2k} show the components of the points $x = (x_B, x_N)$ and $\hat x = (\hat x_B, \hat x_N)$ returned by Algorithm~\ref{algo.MPRA} for an instance of size $n=1000$ and for an instance of size $n=2000$. The set $B$ is $\{1,\ldots,424\}$ in the first instance and it is $\{1,\ldots,1137\}$ in the second instance. 
The large red circle in the plots show the size of $B$. 
For scaling purposes, in both instances the vectors $x$ and  
$\hat x$ are normalized so that $\|x\|_\infty = \|\hat x \|_\infty =1$.  As Figure~\ref{fig:partitioninxands1k} and Figure~\ref{fig:partitioninxands2k} show, in both cases the solutions $x$ and $\hat x$ satisfy the conditions in~\eqref{eq.relint}.

 \begin{figure}[h!]

\centering   
\subfloat[$x\in L\cap\R^n_{++}$]{\label{fig:xL1000}\includegraphics[width=0.3\textwidth]{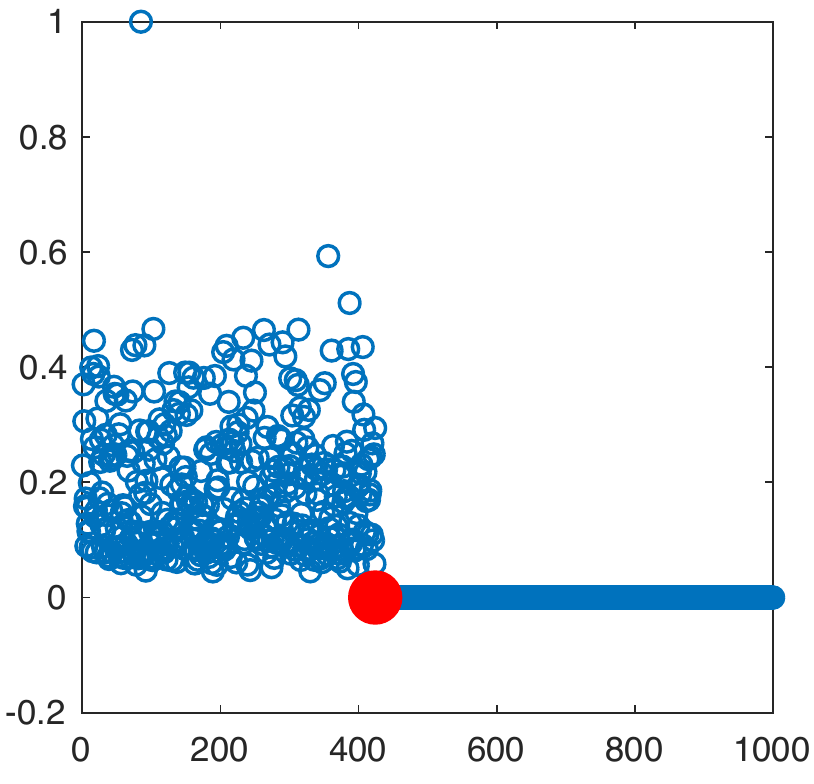}}
\hspace{1in}\subfloat [$\hat x\in L^\perp\cap\R^n_{++}$]{\label{fig:xLperp1000.png}\includegraphics[width=0.3\textwidth]{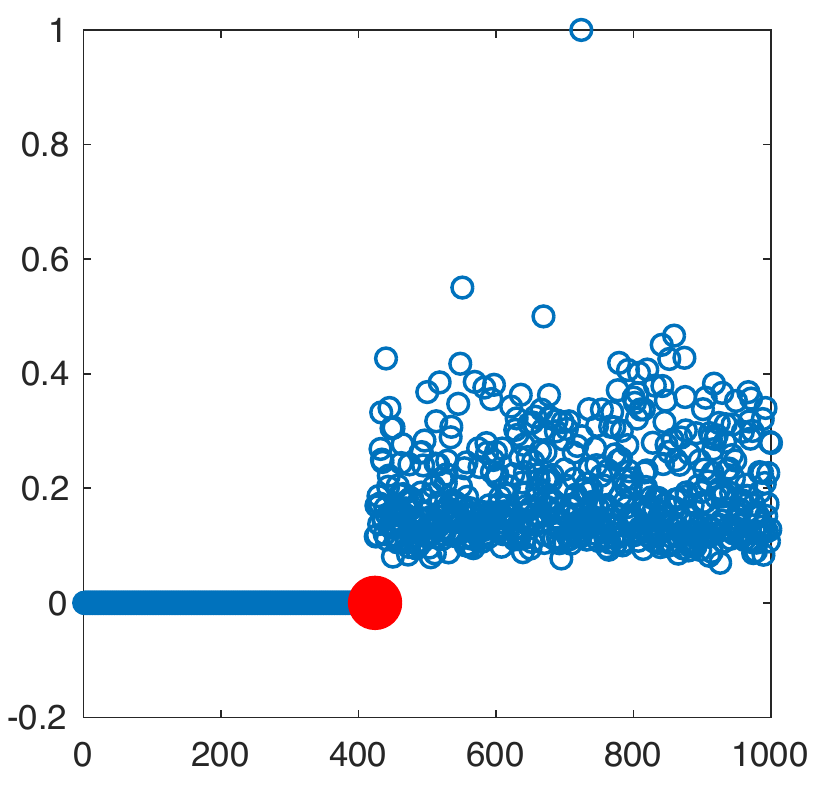}}

\caption{Coordinates of $x\in L\cap\R^n_{+}$ and $\hat x\in L^\perp\cap\R^n_{+}$ found by Algorithm~\ref{algo.MPRA} for some $L \subseteq \R^n$ and $n= 1000$.} \label{fig:partitioninxands1k}
\end{figure}\FloatBarrier

$\:$

 \begin{figure}[h!]

\centering   

\subfloat [$x\in L\cap\R^n_{++}$]{\label{fig:xL2000}\includegraphics[width=0.3\textwidth]{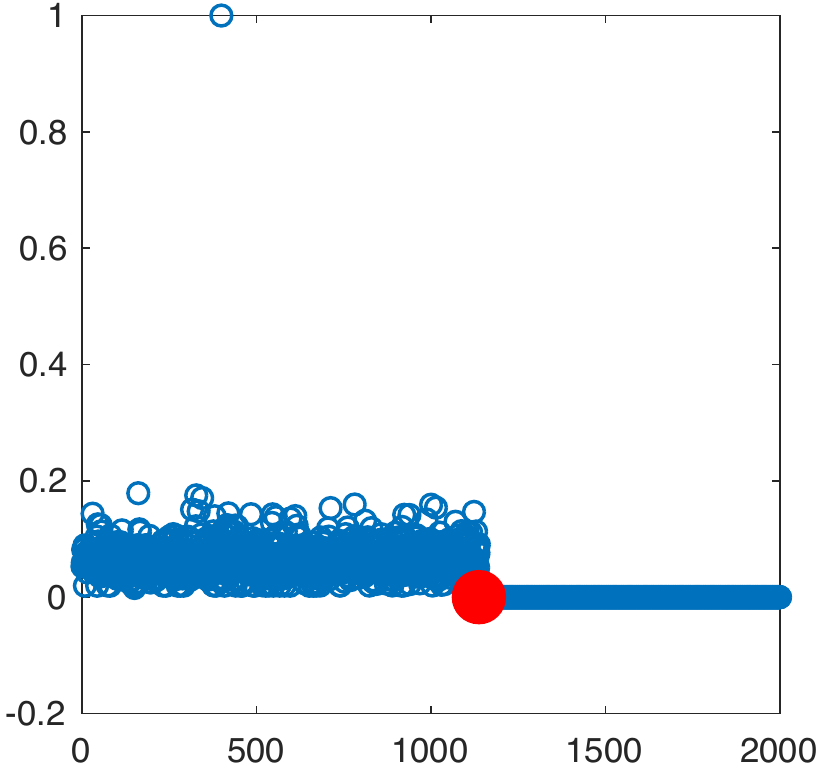}}
\hspace{1in}
\centering \subfloat [$\hat x\in L^\perp\cap\R^n_{++}$]{\label{fig:xLperp2000}\includegraphics[width=0.3\textwidth]{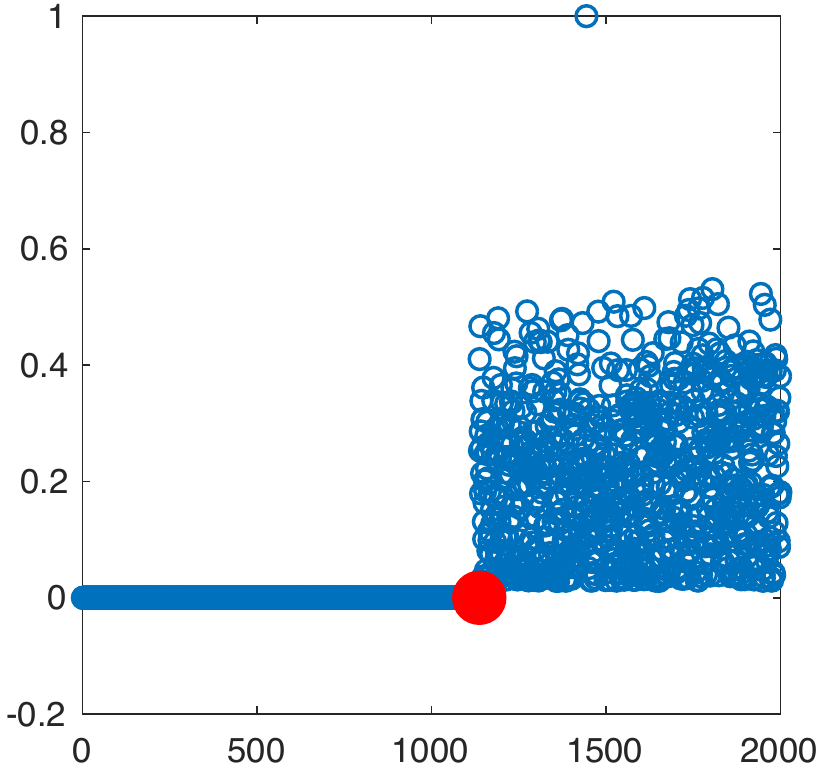}}

\caption{Coordinates of $x\in L\cap\R^n_{+}$ and $\hat x\in L^\perp\cap\R^n_{+}$ found by Algorithm~\ref{algo.MPRA} for some $L\subseteq \R^n$ and $n= 2000$.} \label{fig:partitioninxands2k}
\end{figure}\FloatBarrier

\subsection{Other experiments}

We also performed experiments  to assess the effect of rescaling along multiple directions, as in Algorithm~\ref{algo.MPRA}, versus rescaling only along one direction, as in the original Projection and Rescaling Algorithm in~\cite{PenaS16}.   
More precisely, we compare the performance of Algorithm~\ref{algo.MPRA} versus the modification obtained by changing the update on $D$ and $\hat D$ in Step 5 to  $D = \min\left(\left(I+\diag(e_i)\right)D, U\right)$ and $\hat D = \min\left(\left(I+\diag(e_j)\right)\hat D, U\right)$ where $i$ and $j$ are such that 
for $z_i = \|z\|_\infty$ and $\hat z_j = \|\hat z\|_\infty$.  In most instances the modified version that rescales along one direction  failed to find a solution within a reasonable number (a hundred) of rescaling iterations. 
We  note that~\cite[Section 6]{PenaS16} provides a closed-form formula to update the projection matrix $P$ at low cost after rescaling along one direction. The formula can be extended to handle multiple rescaling directions. However, the formula is computationally attractive only when the number of rescaling directions is small because it requires computing the spectral decomposition of a matrix with rank equal to the number of rescaling directions.   Our numerical experiments indicate that even using the closed-form formula in~\cite{PenaS16} does not compensate for the additional number of iterations required by the modified version of Algorithm~\ref{algo.MPRA} with a single rescaling direction.

\medskip

Table~\ref{table.interior.naive} provides a summary similar to that displayed on Table~\ref{table.interior} of the performance of Algorithm~\ref{algo.MPRA} on a set of naive random instances.  These instances were generated in the same way as those used in the experiments summarized Table~\ref{summaryTableeps1} and Table~\ref{summaryTableeps4}, namely, $L = \ker(A)$ where the entries of $A\in \R^{m\times n}$ are independently drawn from a standard normal distribution.  In contrast to the results summarized in Table~\ref{table.interior} for controlled condition instances, Algorithm~\ref{algo.MPRA} solves most instances easily without rescaling and after a much lower number of total basic iterations.  In particular, Algorithm~\ref{algo.MPRA} solves all naive random instances without rescaling when $m\ne n/2$ and only a few instances require a small number of rescaling steps when $m=n/2$.  For additional illustration of the latter fact, Figure~\ref{fig.rescaling.naive} plots the number of rescaling iterations for the naive random instances with $m=n/2$.
 The last column of Table~\ref{table.interior.naive} shows the fraction of instances where $L\cap \R^n_{++} \ne \emptyset$.  In contrast to the controlled condition instances, this is unknown for naive random instances.  The fraction of instances where $L\cap \R^n_{++} \ne \emptyset$ is consistent with~\eqref{eq.prob} and the subsequent discussion.

\begin{table}[H]
\centering
\caption{Algorithm~\ref{algo.MPRA} on naive random instances}
\begin{tabular}{|c|c|c|c|c|c|}
 \hline
$m$ & $n$& \makecell{Average \# \\ of rescaling\\ iterations}&\makecell{Average total \# \\of  iterations}& \makecell{Average \\ CPU time\\ (in seconds)} & 
\makecell{Fraction of instances \\ with $L\cap \R^n_{++} \ne \emptyset$}\\
\hline
100	  & 200  &	0.496   & 147.094   &	0.0213 & 0.5\\
250  & 500  &	0.338 & 257.660  &0.1428 & 0.496\\
100  & 1000 &	0 & 3.700 &	0.1574 & 1\\
200  & 1000 &	0 & 9.520 &	0.2008 & 1\\
500  & 1000 &	0.228 & 373.352 &	0.7646 & 0.502\\
800  & 1000 &	0   & 5.738 &	0.1175 & 0\\
900  & 1000 &	0   & 2.616 &	0.1156 & 0\\
1000& 2000 &	0.188 & 602.674 & 4.4935 & 0.482\\
\hline
\end{tabular}\label{table.interior.naive}
\end{table}\FloatBarrier

\begin{figure}
\centering
 \includegraphics[width=.7\textwidth] {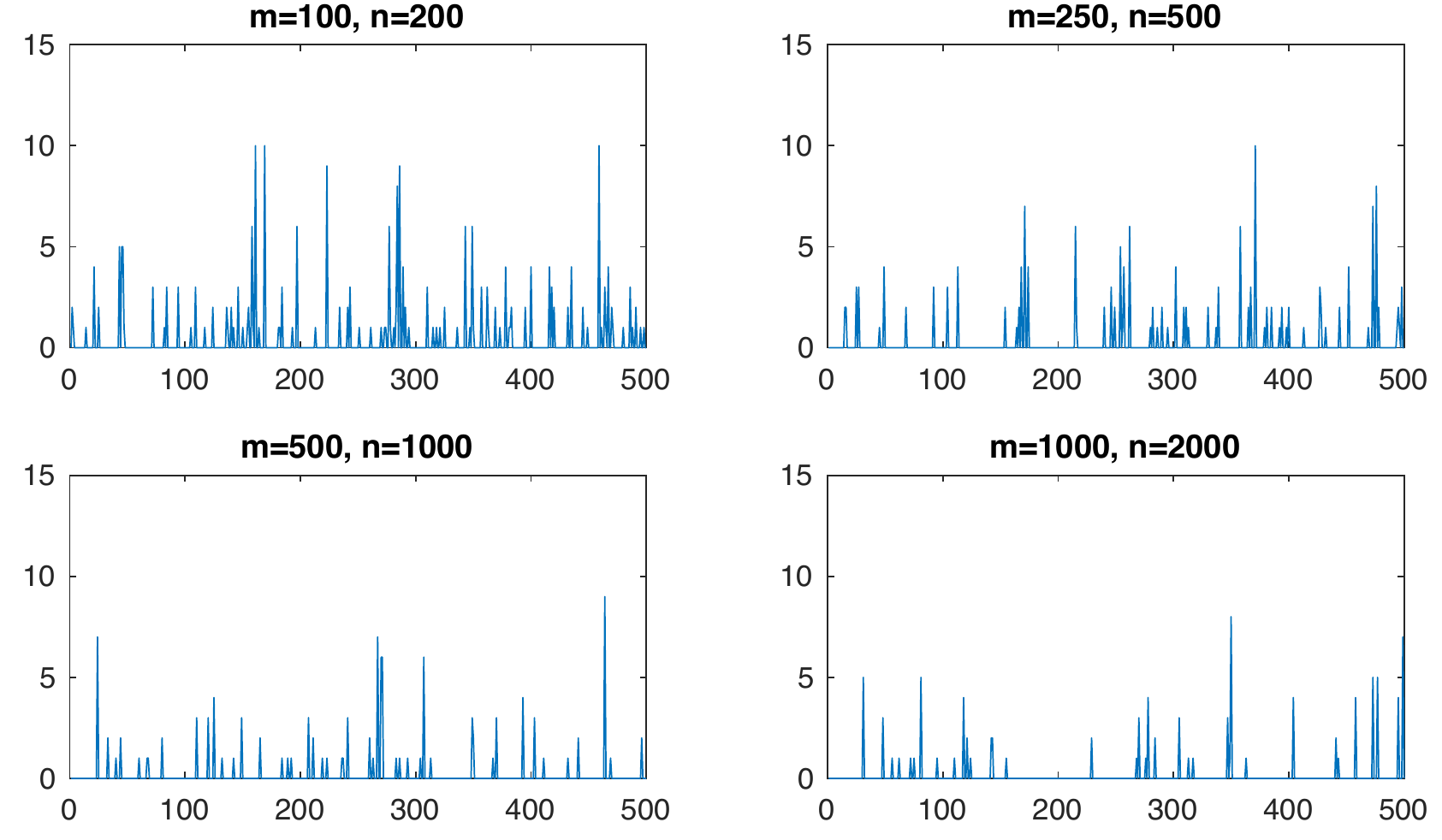}
\caption{Number of rescaling iterations for naive random instances}
\label{fig.rescaling.naive}
\end{figure}

We also compared the performance of  Algorithm~\ref{algo.MPRA} with the state-of-the-art commercial solver CPLEX.  Similar comparisons with Gurobi and MATLAB solvers are reported in~\cite{LiRT15,Roos18}.  Consistent with the reported results in~\cite{LiRT15,Roos18}, we observe that on average Algorithm~\ref{algo.MPRA} is faster than CPLEX by nearly an order of magnitude for problem instances where $L = \ker(A)$ with $A \in \R^{m \times n}$ generated naively at random as in~\cite{LiRT15,Roos18} and as in the set of experiments summarized in Table~\ref{table.interior.naive}.  On the other hand, the difference in speed is about the opposite, that is, CPLEX is nearly an order of magnitude faster when $A\in \R^{m\times n}$ is generated so that $L\cap\R^n_{++}$ has a controlled condition measure via the procedure suggested by Proposition~\ref{prop.control.cn} as in the set of experiments summarized in Table~\ref{table.interior}. We attribute this sharp difference to the fact that the naively generated instances are generally easier and can usually be solved within one single round of basic procedure and without the need for rescaling even when $m = n/2$ as Table~\ref{table.interior.naive} illustrates.  By contrast, instances with controlled condition measure, such as those in the set of experiments summarized in Table~\ref{table.interior}, are significantly more challenging and require on average ten or more rounds of basic and rescaling steps. We note that for similarly generated instances, the numerical experiments reported in~\cite{Roos18} generally require several rescaling iterations when $m = n/2$ while our algorithm solves most of these instances without any rescaling iterations. This difference is likely due to the different basic procedures used in~\cite{Roos18} and in our numerical experiments. The rescaling method in~\cite{Roos18} uses a variant of the von Neumann algorithm as its basic procedure while we use the smooth perceptron scheme.


\section{Concluding remarks}

We have described a computational implementation and numerical experiments of an  Enhanced Projection and Rescaling algorithm for finding most interior solutions to the feasibility problems 
\[  
\text{find}  \; x\in L\cap\R^n_{+} \;\;\;\; \text{ and } \; \;\;\;\;
\text{find} \; \hat x\in L^\perp\cap\R^n_{+},
\]
where $L$ denotes a linear subspace in $\R^n$ and $L^\perp$ denotes its orthogonal complement.  Our numerical results provide promising evidence of the effectiveness of this algorithmic approach.

The MATLAB code for our implementation is comprised of a set of  MATLAB functions with verbatim implementations of Algorithm~\ref{algo.MPRA} through Algorithm~\ref{alg:smooth}.  Our MATLAB code is publicly available at the following website

\begin{center} {\tt http://www.andrew.cmu.edu/user/jfp/epra.html} \end{center} 

{   The tables presented in this paper were created by averaging the results obtained from running the following MATLAB functions.  
\begin{itemize}
\item {\tt TestSimpleBasicProcedures(m,n,N,epsilon):} This is the code used to generate and test the set of instances summarized in each row of Table~\ref{summaryTableeps1}, Table~\ref{summaryTableeps4}, and Table~\ref{new.table1}.
\item {\tt TestControlledConditionBasicProcedures(m,n,N,epsilon,delta):} This is the code used to generate and test the set of instances summarized in each row of Table~\ref{summaryTableeps1.control}, Table~\ref{summaryTableeps4.control}, and Table~\ref{new.table2}.   
\item {\tt TestControlledConditionRescaled(m,n,N,delta):} This is the code used to generate
 and test the set of instances summarized in each row of Table~\ref{table.interior}.

\item {\tt TestPartitionRescaled(n,N):} This is the code used to generate  and test the set of instances summarized in each row of Table~\ref{table.partition}.

\item {\tt TestSimpleRescaled(m,n,N):} This is the code used to generate  and test the set of instances summarized in each row of Table~\ref{table.interior.naive}.  This code also compares the performance of Algorithm 1 with a modified version that rescales along one direction only including a more efficient update on the projection matrix after each rescaling step.
\end{itemize}
The input parameters for the above functions are as follows
\begin{itemize}
\item {\tt N:} Number of instances.  We used {\tt N = 1000} in Table~\ref{summaryTableeps1}
through
Table~\ref{summaryTableeps4.control}, and {\tt N = 100} in Table~\ref{new.table1} and Table~\ref{new.table2}.  We used {\tt N = 500} in Table~\ref{table.interior} through Table~\ref{table.interior.naive}.
\item {\tt m:} Number of rows of $A\in \R^{m\times n}$ such that $L = \ker(A)$
\item {\tt n:} Dimension of the ambient space
\item {\tt epsilon:} Rescaling condition parameter
\item {\tt delta:} Upper bound on the values of a subset of randomly chosen positive entries of the most central solution.  The smaller {\tt delta,} the more ill-conditioned the problem.
We used {\tt delta = 0.001} for the experiments summarized in Table~\ref{summaryTableeps1.control}, Table~\ref{summaryTableeps4.control}, 
Table~\ref{new.table2}, and Table~\ref{table.interior}. 
\end{itemize}
Algorithm~\ref{algo.MPRA} through Algorithm~\ref{alg:smooth} are implemented via the following MATLAB functions.
\begin{itemize}
\item {\tt MultiEPRA(A,AA,n,z0,U):} This code implements Algorithm~\ref{algo.MPRA}. Assume $L = \ker({\tt A})\subseteq \R^n$ and $L^\perp = \ker({\tt AA})$.  Use {\tt z0} as starting point for the basic procedure and {\tt U} to upper bound the rescaling matrices. 
\item {\tt perceptron(P,z0,epsilon):} This code implements Algorithm~\ref{alg:perceptron}.
\item {\tt VN(P,z0,epsilon):} This code implements Algorithm~\ref{alg:vonNeumann}.
\item {\tt VNA(P,z0,epsilon):} This code implements Algorithm~\ref{alg:vonNeumann.away}.
\item {\tt smooth(P,u0,epsilon):} This code implements Algorithm~\ref{alg:smooth}
\end{itemize}.

}

\bibliographystyle{plain}

\end{document}